\documentclass[a4paper,10pt]{amsart}

\usepackage[latin1]{inputenc}
\usepackage{amsmath}
\usepackage{amsfonts}
\usepackage{amssymb}
\usepackage[T1]{fontenc}
\usepackage{dsfont}
\usepackage[all]{xy}
\usepackage{stmaryrd}

\newtheorem{theorem}{Theorem}[section]
\newtheorem{prop}[theorem]{Proposition}
\newtheorem{lem}[theorem]{Lemma}
\newtheorem{corol}[theorem]{Corollary}

\theoremstyle{definition}
\newtheorem{defi}[theorem]{Definition}

\theoremstyle{definition}

\newtheorem*{theoremmain}{Corollary \ref{theorem:main}}
\newtheorem*{theoremmaintwo}{Theorem \ref{theorem:main2}}

\def\Hom{{\rm{Hom}}}
\def\Ext{{\rm{Ext}}}

\def\add{{\rm{add}\,}}
\def\Gr{{\rm{Gr}}}
\def\End{{\rm{End}}}
\def\dim{{\rm{dim}\,}}
\def\ddim{{\textbf{dim}\,}}
\def\ind{{\rm{ind-}}}

\def\Ob{{\rm{Ob}}}
\def\Cl{{\rm{Cl}}}

\def\C{{\mathbb{C}}}
\def\Z{{\mathbb{Z}}}
\def\N{{\mathbb{N}}}
\def\Q{{\mathbb{Q}}}

\def\<{\left<}
\def\>{\right>}

\def\ens#1{\left\{ #1 \right\}}
\def\fl{{\longrightarrow}\,}

\title{Caldero-Keller approach to the denominators of cluster variables}
\author{\textsc{G. Dupont}}
 \address{Universit\'e de Lyon \\
Universit\'e Lyon 1 \\
Institut Camille Jordan CNRS UMR 5208 \\
43, boulevard du 11 novembre 1918\\
F-69622 Villeurbanne Cedex.}
 \email{dupont@math.univ-lyon1.fr}

\begin{document}
\maketitle

\begin{abstract}
	Buan, Marsh and Reiten proved that if a cluster-tilting object $T$ in a cluster category $\mathcal C$ associated to an acyclic quiver $Q$ satisfies certain conditions with respect to the exchange pairs in $\mathcal C$, then the denominator in its reduced form of every cluster variable in the cluster algebra associated to $Q$ has exponents given by the dimension vector of the corresponding module over the endomorphism algebra of $T$. In this paper, we give an alternative proof of this result using the Caldero-Keller approach to acyclic cluster algebras and the work of Palu on cluster characters.
\end{abstract}

\setcounter{tocdepth}{1}
\tableofcontents

\begin{section}{Introduction}
	
	Cluster algebras were introduced in \cite{cluster1} in order to study problems of total positivity. Then, they were subject to developments in various directions including combinatorics, Lie theory, Teichm\"uller theory and quiver representations.

	By definition, the cluster algebras are commutative algebras generated by a set of variables called \emph{cluster variables} gathered into sets of fixed cardinality called \emph{clusters}. The \emph{Laurent phenomenon} proves that the elements in the cluster algebra are Laurent polynomials in the variables contained in any fixed cluster. The initial data for constructing a cluster algebra is a pair $(Q,\textbf u)$ where $Q$ is a quiver and $\textbf u=(u_i, i \in Q_0)$ is a tuple of indeterminates over $Q$, the associated cluster algebra is denoted by $\mathcal A(Q)$.

	The link with the representation theory of quivers found one of its motivation in the finite type classification of \cite{cluster2}. This asserts that the number of cluster variables in $\mathcal A(Q)$ is finite if and only if $Q$ is a Dynkin quiver and in this case there is a 1-1 correspondence between the \emph{almost positive roots} system of $Q$ (that is the disjoint union of the positive roots and the opposite of simple roots) and the cluster variables in $\mathcal A(Q)$. Considering that the cluster variables are Laurent polynomials, one can define the denominator vector of a cluster variable $P(\textbf u)/\prod_i u_i^{d_i}$ in reduced form as the tuple $\textbf d=(d_i)$. The authors obtained a first interesting description of the denominator vectors for cluster variables when $Q$ is a certain quiver of Dynkin type proving that the denominator vectors for the cluster variables were the almost positive roots of $Q$ (see also \cite{CCS1}). If $Q$ is an acyclic quiver (that is without oriented cycles), the cluster algebra $\mathcal A(Q)$ is called an acyclic cluster algebra.

	Initiated by the work of \cite{MRZ}, the research of a theoretical framework for the study of cluster algebras leaded to a fruitful categorification with the cluster category  introduced in \cite{BMRRT} (see also \cite{CCS1} for the Dynkin type $\mathbb A$). In \cite{BMRT}, a surjective map $\alpha$ was defined from the set of cluster variables of $\mathcal A(Q)$ to the set of indecomposable exceptional objects in the cluster category $\mathcal C$. This map satisfies that the denominator of a cluster variable $x$ can be described in terms of dimension vector (in a proper sense to be defined) of the indecomposable exceptional object $\alpha(x)$. In \cite{CC}, a map was defined in the reverse direction, allowing the authors of \cite{CK1,CK2} to prove independently the interpretation of the denominators of cluster variables in terms of composition factors of the corresponding exceptional object in the cluster category. In particular, this generalizes the correspondence between denominator vectors of cluster variables and almost positive roots of $Q$ in the case when $Q$ is a Dynkin quiver.

	It is an interesting question to wonder how the denominator behaves if one expresses the cluster variables in terms of different seeds. It is known since \cite{BMRT,CK2} that each seed can be associated with a cluster-tilting object in $\mathcal C$. A description of the denominators was given in \cite{CCS1,CCS2} when $Q$ is a Dynkin quiver. Recently, Buan, Marsh and Reiten obtained a generalization in \cite{BMR3} when $Q$ is any acyclic quiver under some conditions on the cluster-tilting object corresponding to the seed in which the cluster variables are expressed.

	It is known since \cite{BCKMRT} that the approaches developed independently in \cite{BMRT} and \cite{CC,CK1,CK2} are dual to each other. Namely, the maps defined there are their mutual inverse. These two approaches are very complementary and allow to have a very good understanding of the link between cluster algebras and cluster categories. 

	The recent works of \cite{Palu}, based on conjectures of \cite{CK1}, developed a Caldero-Keller approach to the change of seeds in cluster algebras. By reading \cite{BMR3}, one could realize that the given conditions on the cluster-tilting objects appeared naturally in the context of the Caldero-Keller approach to the change of seeds. Following this idea, we found it interesting to give alternative proofs for some results of \cite{BMR3} using the Caldero-Keller approach and the works of Palu. Except for the last section, our results are independent of those of \cite{BMR3}.

	The paper is organized as follows. In section \ref{section:background} we present the necessary background and the main result of this paper. In section \ref{section:character}, we recall the definition and properties of the generalized Caldero-Chapoton map from \cite{Palu}. In section \ref{section:weakly}, we recall some essential property for the behaviour of denominator vectors of Laurent polynomials, called the \emph{positivity condition} in \cite{BMRT} and \emph{weak positivity} in \cite{CK2}. The last two sections are devoted to the proofs of our results.
	
\end{section}

\begin{section}{Background and main results}\label{section:background}
	
		Let $Q=(Q_0,Q_1)$ be an acyclic quiver with $Q_0=\ens{1, \ldots, q}$ is the set of vertices and $Q_1$ is the set of arrows. We will always assume that $Q$ is connected (ie the underlying unoriented graph is connected). In all the paper $k$ denotes the the field $\C$ of complex numbers. We denote by $kQ$ the corresponding path algebra, $kQ$-mod the category of finite dimensional $kQ$-modules and $D^b(kQ)$ the bounded derived category of finite dimensional $kQ$-modules. $\tau$ denotes the Auslander-Reiten translation on $D^b(kQ)$, $S$ the shift functor and $F=\tau^{-1}S$. Let $\mathcal C=D^b(kQ)/F$ be the cluster category associated to $Q$, that is the orbit category of $F$ on $D^b(kQ)$. 
		
		We denote by $\Ob(\mathcal C)$ the set of objects in $\mathcal C$ and by $\ind \mathcal C$ the set of indecomposable objects in $\mathcal C$. An object will be called \emph{exceptional} if it has no self-extension, \emph{basic} if all its distinct indecomposable direct summands are non isomorphic. An object $T$ will be called a \emph{cluster-tilting object} if it is exceptional with $q$ non-isomorphic indecomposable direct summands. We set 
		$$\mathcal T=\ens{M \in \ind \mathcal C \ : \ \Ext^1_{\mathcal C}(M,M)=0}.$$

	 	Fix $\textbf u=(u_1, \ldots, u_q)$ a $q$-tuple of indeterminates over $\Q$. We will denote by $\mathcal A(Q)$ the (coefficient free) cluster algebra over with initial seed $(\textbf u, Q)$. We denote by $\Cl(Q)$ the set of cluster variables in $\mathcal A(Q)$.

	\begin{defi}
		Write $F=P(\textbf u)/\prod_{i=1}^qu_i^{d_i}$ a Laurent polynomial in the $u_i$. Assume that $F$ is written in its irreducible form, that is such that $P$ is not divisible by any $u_i$ and $d_i \in \Z$ for every $i$. The the $q$-tuple $(d_1, \ldots, d_q)$ is called the \emph{denominator vector} of $F$ and is denoted by $\delta(F)$.
	\end{defi}

		In \cite{BMRT}, the authors introduced a surjective map $\alpha: \Cl(Q) \fl \mathcal T$ from the set of cluster variables to the set of indecomposable exceptional objects in the cluster category $\mathcal C$. It is defined by 
		$$\alpha: \left\{\begin{array}{rcl}
			\Cl(Q) & \fl & \mathcal T \\
			x & \mapsto & \alpha(x)
		\end{array}\right.$$
		where $\alpha(x)$ is the unique indecomposable exceptional object with dimension vector $\delta(x)$ if $x \neq u_i$ for all $i \in Q_0$ and $\alpha(x)=SP_i$ if $x=u_i$ for some $i \in Q_0$.

	 	Another point of view, which is dual to this one, consists of realizing the cluster algebra from the cluster category. This approach was developed in \cite{CC,CK1,CK2} where the authors defined and studied a map $X_?:\Ob(\mathcal C) \fl \Z[u_i^{\pm1}, \ldots, u_q^{\pm 1}]$ called the \emph{Caldero-Chapoton map}. 
	 	
	 	In \cite{CK2}, the authors proved that the Caldero-Chapoton map induces a bijection between $\mathcal T$ and $\Cl(Q)$ and it turned out that the induced map on $\mathcal T$ is the bijection inverse to $\alpha$ (see \cite{BCKMRT}), namely
		$$\alpha(X_M)=M \textrm{ and } X_{\alpha(x)}=x$$
		for any $M \in \mathcal T$ and $x \in \Cl(Q)$.
	
		From now on, we fix $T=\bigoplus_{i=1}^q T_i$ a  cluster-tilting object in $\mathcal C$. We denote by $Q_T$ the quiver of the cluster-tilted algebra $B=\End_{\mathcal C}(T)$. It is known (see \cite{BMRT}) that there is a seed $(\textbf x, Q_T)$ mutation-equivalent to $(\textbf u, Q)$ in $\mathcal A(Q)$ where $\textbf x=(x_1, \ldots, x_q)$ is a $q$-tuple of indeterminates over $\Q$. We denote by $\mathcal A(Q_T)$ the coefficient-free cluster algebra with initial seed $(\textbf x, Q_T)$. 
		
		As each $x_i$ is a Laurent polynomial in the $u_i$, we can set
		$\Phi_T: \mathcal A(Q_T) \fl \mathcal A(Q)$ the canonical algebra homomorphism sending $x_i$ to its expansion in $\Z[u_1^{\pm 1}, \ldots, u_q^{\pm 1}]$. Then, $\Phi_T$ is an algebra isomorphism from $\mathcal A(Q_T)$ to $\mathcal A(Q)$ inducing a bijection from $\Cl(Q_T)$ to $\Cl(Q)$.

		If $F$ is a Laurent polynomial in the $x_i$, $i \in Q_0$, we denote by $\delta_T$ the denominator vector of $F$ expressed in $\textbf x$.
	
		We can draw the following picture where the top maps are inverse bijections:
	 	$$\xymatrix{ 
			\mathcal T \ar@<+2pt>[r]^{X_?} & \ar@<+2pt>[l]^{\alpha} \Cl(Q) \ar[r]^{\delta} & \Z^{Q_0}\\
				& \Cl(Q_T)\ar[u]^{\Phi_T} \ar[r]_{\delta_T} & \Z^{Q_0}
		}$$

		It is proved in \cite{CK2} that for every object $M \in \mathcal T$ non-isomorphic to any $SP_i$, the denominator vector of $X_M$ is $$\delta(X_M)=\ddim M=(\dim \Hom_{\mathcal C}(P_1, M), \ldots, \dim \Hom_{\mathcal C}(P_q, M)).$$
		Recently, in \cite{BMR3}, the authors proved that under certain conditions on $T$, namely the so-called \emph{exchange compatibility} (see section \ref{section:compatibility} for a definition), one could obtain a generalization of the above result, also generalizing the denominator theorem of \cite{BMRT}. Following the authors, we set:
		
		\begin{defi}\label{defi:Tdenom}
			A cluster variable $x$ in $\mathcal A(Q_T)$ is said to \emph{have a $T$-denominator} if $$\delta_T(x)=(\dim \Hom_{\mathcal C}(T_1, \alpha(\Phi_T(x))), \ldots, \dim \Hom_{\mathcal C}(T_q, \alpha(\Phi_T(x)))$$
			if $x \not \in \ens{x_1, \ldots, x_q}$ and $\alpha(\Phi_T(x_i))=ST_i$ for every $i \in Q_0$.
		\end{defi}
		
		The following was proved in \cite{BMR3} :
		\begin{theorem}[\cite{BMR3}]\label{theorem:mainBMR3}
			Let $Q$ be a finite quiver with no oriented cycles, let $\mathcal C$ be the cluster category associated to $kQ$ and let $T=\bigoplus_{i=1}^qT_i$ be a cluster-tilting object in $\mathcal C$ such that each $T_i$ is exchange compatible, then every cluster variable has a $T$-denominator.
		\end{theorem}
	 
		We propose an alternative proof of this result using the Caldero-Chapoton-Keller approach in a generalized context, relying on the works of Palu in \cite{Palu}. We denote by $X^T_?: \mathcal C \fl \mathcal A(Q_T)$ the so-called cluster character on $\mathcal C$ associated to the cluster-tilting object $T$ as introduced in \cite{Palu} (see section \ref{section:character} for definitions). It is known that $X^T_?$ induces a bijection between $\mathcal T$ and the set $\Cl(Q_T)$ of cluster variables in $\mathcal A(Q_T)$. The situation can be described in the following commutative diagram :
		$$\xymatrix{ 
			\mathcal T \ar[rd]_{X^T_?} & \ar[l]^{\alpha} \Cl(Q)\\
				& \Cl(Q_T)\ar[u]^{\Phi_T} \ar[r]^{\delta_T} & \Z^{Q_0}
		}$$
		
		Our main result is the following :
		\begin{theorem}\label{theorem:main}
			Let $Q$ be a finite quiver with no oriented cycles. Let $\mathcal C$ be the cluster category associated to $kQ$ and let $T=\bigoplus_{i=1}^qT_i$ be a cluster-tilting object in $\mathcal C$ such that each $T_i$ is exchange compatible, then for any indecomposable exceptional object $M$ of $\mathcal C$
			$$\delta_T(X^T_M)=\left\{\begin{array}{rl}
				-e_i & \textrm{ if } M\simeq ST_i\\
				\ddim \Hom_{\mathcal C}(T,M)& \textrm{ otherwise}
			\end{array}\right.$$
			In particular, every cluster variable has a $T$-denominator.
		\end{theorem}
		This theorem will be proved  in section \ref{section:compatibility}. Our proof is independent of the works of \cite{BMR3}. We will see in section \ref{section:compatibility} that theorems \ref{theorem:main} and \ref{theorem:mainBMR3} are equivalent.
	
 		We also propose an alternative proof of the condition of exchange compatibility. Namely, using the Caldero-Chapoton-Keller approach, we give an alternative proof of the point (b) of Theorem 1.5 in \cite{BMR3}:
		\begin{theorem}\label{theorem:main2}
			Let $Q$ be a finite quiver with no oriented cycles and $\mathcal C$ be the cluster category associated to $kQ$. Let $T=\bigoplus_{i=1}^q T_i$ be a cluster-tilting object in $\mathcal C$. Let
			$\mathcal A(Q)$ be the cluster algebra associated to $Q$. If every cluster variable of $\mathcal A(Q)$ has a $T$-denominator, then $\End_{\mathcal C}(T_i)\simeq k$ for all $i \in Q_0$.
		\end{theorem}
		This theorem will be proved in section \ref{section:main2}.
\end{section}

\begin{section}{Cluster characters for cluster categories}\label{section:character}
	In this section $Q=(Q_0,Q_1)$ still denotes an acyclic quiver and $\mathcal C$ is the associated cluster category endowed with the shift functor $S$. We denote by $\mathcal A(Q)$ the cluster algebra with initial seed $(Q,\textbf u)$ and by $\Cl(Q)$ the set of cluster variables in $\mathcal A(Q)$. 
		
	Write $X_?$ the Caldero-Chapoton map on $\mathcal C$ (see \cite{CC}).
	\begin{theorem}[\cite{CK2}]
		Let $Q$ be an acyclic quiver, then $X_?$ induces a bijection from $\mathcal T$ to $\Cl(Q)$.
	\end{theorem}

	Fix now a basic tilting object $T$ and denote by $B=\End_{\mathcal C}(T)$ and $Q_T$ the quiver of $B$. We denote by $\mathcal A(Q_T)$ the cluster algebra with initial seed $(Q_T,\textbf x)$ and by $\Phi_T$ the canonical algebra homomorphism 
	$$\Phi: \mathcal A(Q) \fl \mathcal A(Q_T)$$
	sending $x_i$ to its Laurent expansion in the cluster $\textbf u$.
	In particular, $\Phi$ induces a bijection from $\Cl(Q_T)$ to $\Cl(Q)$.

	We denote by $F=\Hom_{\mathcal C}(T,?): \mathcal C \fl B\textrm{-mod}$. For any objects $M,N$ in $B$-mod, we write
	$$\<M,N\>=\dim \Hom_B(M,N) - \dim \Ext^1_B(M,N)$$
	$$\<M,N\>_a=\<M,N\>-\<N,M\>$$
	
	Note that in general, $\<-,-\>$ does not induce a bilinear form on the Grothendieck group $K_0(B)$. Nevertheless, Palu proved in \cite{Palu} that the anti-symmetrized bilinear form $\<-,-\>_a$ induces a form on $K_0(B)$.

	For any object $M$ in $\mathcal C$, following \cite{Palu}, we define
	$$X_M^T=\left\{ \begin{array}{ll}
		x_i & \textrm{ if }M\simeq ST_i\\
		\sum_{\textbf e} \chi(\Gr_{\textbf e}(FM)) \prod_i x_i^{\<S_i, \textbf e\>_a-\<S_i, FM\>} & \textrm{ otherwise}
	\end{array}\right.$$
	where $\textbf e$ runs over the dimension vectors of $B$-modules, $\Gr_{\textbf e}(FM)$ denotes the variety of submodules of $FM$ with dimension vector $\textbf e$ and $\chi$ denotes the Euler-Poincar\'e characteristic. Moreover, this map satisfies 
	$$X^T_{M \oplus N}=X^T_MX^T_N$$
	for any two objects $M,N$ in $\mathcal C$.

	Then, it is known that the following diagram commutes
	$$\xymatrix{
		& \mathcal T \ar[ld]_{X_?} \ar[rd]^{X_?^T} \\
		\mathcal A(Q) & & \ar[ll]^{\Phi_T} \mathcal A(Q_T)
	}$$
	In particular, $X_?^T$ induces a bijection from $\mathcal T$ to $\Cl(Q_T)$.

	For any two objects $U,V$ such that $\Ext^1_{\mathcal C}(U,V) \simeq k$, we denote by $E_{U,V}$ and $E_{V,U}$ the unique objects in $\mathcal C$ such that there exists non-split triangles
	$$U \fl E_{V,U} \fl V \fl SU$$
	$$V \fl E_{U,V} \fl U \fl SV$$

	One of the interesting properties of the cluster character $X^T_?$ is its behaviour under multiplication:

	\begin{theorem}[\cite{Palu}]\label{theorem:multiplication}
		Fix $Q$ an acyclic quiver, $\mathcal C$ its cluster category. Fix $U,V$ two objects in $\mathcal C$ such that $\dim \Ext^1_{\mathcal C}(U,V)=1$, then
		$$X^T_UX^T_{V}=X^T_{E_{U,V}} + X^T_{E_{V,U}}.$$
	\end{theorem}

	In the sequel, we will mainly apply this theorem to a particular case of pairs of objects $(U,V)$ such that $\dim \Ext^1_{\mathcal C}(U,V) \simeq k$, namely when $(U,V)$ is an exchange pair in the sense of \cite{BMRRT}. We recall here the definition of an exchange pair.
	
	Fix $U_{0}$ an indecomposable exceptional object in $\mathcal T$, according to \cite{BMRRT}, it can be completed into a cluster-tilting object $U=\overline U \oplus U_{0}$. Now it is known that there exists an unique $U_{0}^*$ non isomorphic to $U_{0}$ such that $\overline U \oplus U_{0}^*$ is a cluster-tilting object in $\mathcal C$. The pair $(U_{0},U_{0}^*)$ is called an \emph{exchange pair}.
	
	For such an exchange pair, $\Ext^1_{\mathcal C}(U_{0},U_{0}^*) \simeq k$ and $E_{U_0,U_0^*}, E_{U_0^*,U_0}$ are objects in $\add \overline U$, where $\add \overline U$ denotes the subcategory of objects whose direct summands are direct summands of $\overline U$. In particular $E_{U_0,U_0^*}$ and $E_{U_0^*,U_0}$ are exceptional objects.

	As an immediate corollary, we obtain:	
	\begin{corol}\label{corol:exchangemult}
	 	Fix $(U,U^*)$ an exchange pair in $\mathcal C$, then 
	 	$$X^T_UX^T_{U^*}=X^T_{E_{U,U^*}} + X^T_{E_{U^*,U}}$$
	\end{corol}
	
	We recall the definition of the \emph{tilting graph} of a cluster category. The vertices of the tilting graph are representatives of the isomorphism classes of cluster-tilting objects of $\mathcal C$ and there is an edge between two cluster-tilting objects $M$ and $M'$ if $M$ and $M'$ differ only by one direct summand, that is if $M=U \oplus \overline U$ and $M'=U^* \oplus \overline U$ where $(U,U^*)$ is an exchange pair in $\mathcal C$. It is proved in \cite{BMRRT} (see also \cite{HU}) that the tilting graph of a cluster category is connected.
	
\end{section}

\begin{section}{Weakly positive Laurent polynomials}\label{section:weakly}
	An essential notion concerning the behaviour of the denominators of Laurent polynomials under sum and multiplication is the \emph{positivity condition}, introduced in \cite{BMRT}. Following \cite{CK2}, we will call this property the \emph{weak positivity}. 
	
	\begin{defi}
		A Laurent polynomial $F=P(\textbf x)/\textbf x^{\textbf d}$ in an irreducible form will be called \emph{weakly positive} if $P(\textbf z)>0$ for every $\textbf z \in \N^{Q_0}$ with at most one vanishing component.
	\end{defi}
	
	Given two vectors $\textbf d, \textbf e \in \Z^{Q_0}$, we write
	$$\max(\textbf d, \textbf e)=(\max(d_i,e_i))_{i \in Q_0}$$

	This notion is known to be useful in considering the denominator vectors of cluster variables, it appears to be a central point for the techniques developed in \cite{BMRT,BMR3,CK2}. The following gives the essential reasons for introducing the weak positivity.

	\begin{lem}\cite{CK2}\label{lem:faiblementpositif}
		We fix any Laurent polynomial ring $R$. For any element $L \in R$, denote by $\delta_R$ the denominator vector of $L$. Then,
		\begin{enumerate}
			\item If $L_1$ and $L_2$ are weakly positive Laurent polynomials, then $L_1+L_2$ is also weakly positive. Moreover, 
			$$\delta_R(L_1 + L_2) =\max(\delta_R(L_1), \delta_R(L_2))$$
			\item If $L_1$ and $L_2$ are Laurent polynomials such that $L_1$ is weakly positive, then $L_2$ is weakly positive if and only if $L_1L_2$ is weakly positive. Moreover
			$$\delta_R(L_1 L_2) =\delta_R(L_1)+\delta_R(L_2)$$
		\end{enumerate}
	\end{lem}
 	
	In \cite{CK2}, the authors proved that for any indecomposable exceptional object, $X_M$ is a weakly positive Laurent polynomial. We prove similarly that $X^T_M$ is weakly positive:

	\begin{lem}
		For any indecomposable exceptional object $M$ in $\mathcal C$ then $X^T_M$ is a weakly positive Laurent polynomial in $\textbf x$.
	\end{lem}
	\begin{proof}
		As $M$ is an indecomposable exceptional object, it can be completed into a cluster-tilting object in $\mathcal C$. The tilting graph being connected, it suffices to prove by induction that if the $X^T_U$ are weakly positive for all the direct summands $U$ of a cluster-tilting object $R$, then the $X^T_U$ are weakly positive for all the direct summands $U$ of any cluster-tilting object $R'$ joined to $R$ by an edge in the tilting graph. 
		
		We start the induction from the cluster-tilting object $ST=\bigoplus_{i=1}^q ST_i$. For every $i \in Q_0$, $X^T_{ST_i}=x_i=1/x_i^{-1}$ is a weakly positive Laurent polynomial. 
		
		Fix now $R$ and $R'$ two cluster-tilting objects joined by an edge in the tilting graph. The $R=U \oplus \overline U$ and $R'=U^* \oplus \overline U$ where $(U,{U^*})$ an exchange pair. We assume that $X^T_V$ is a weakly positive Laurent polynomial for all the direct summands $V$ of $R$. According to corollary \ref{corol:exchangemult}, we have
		$$X^T_UX^T_{{U^*}}=X^T_{E_{U,U^*}} + X^T_{E_{U^*,U}}.$$

		According to lemma \ref{lem:faiblementpositif}, if $X^T_{E_{U,U^*}},X^T_{E_{U^*,U}}$ and $X^T_U$ are weakly positive then so is $X^T_{{U^*}}$. As $E_{U,U^*},E_{U^*,U} \in \add(\overline U)$, $X^T_{E_{U,U^*}}$, $X^T_{E_{U^*,U}}$ and $X^T_{U}$ are weakly positive Laurent polynomials by induction and then $X^T_{U^*}$ is weakly positive.
	\end{proof}
	
	\begin{corol}\label{lem:XTMweak}
		For any object $M$ without self-extensions, $X^T_M$ is a weakly positive Laurent polynomial in $\textbf x$.
	\end{corol}
	\begin{proof}
		Write $M=\bigoplus_i M_i$ the decomposition into indecomposable summands, then $X^T_M=\prod_i X^T_{M_i}$ is a product of weakly positive Laurent polynomials and is thus weakly positive.
	\end{proof}	
\end{section}

\begin{section}{Proof of theorem \ref{theorem:main}}\label{section:compatibility}
	
	In this section, we prove theorem \ref{theorem:main}. As we will see, the condition of \emph{exchange compatibility} introduced in \cite{BMR3} arises naturally as the necessary condition for cluster variables to have a $T$-denominator. We first recall the definition of exchange compatibility from \cite{BMR3}:

	\begin{defi}
		Fix $(U,{U^*})$ an exchange pair. An indecomposable exceptional object $N$ in $\mathcal C$ is called \emph{compatible with the exchange pair $(U,{U^*})$} if whenever $U \not \simeq \tau N \not \simeq {U^*}$ we have
		$$\dim \Hom_{\mathcal C}(N,U)+\dim \Hom_{\mathcal C}(N,{U^*})=\max(\dim \Hom_{\mathcal C}(N,E_{U,U^*}),\dim \Hom_{\mathcal C}(N,E_{U^*,U}))$$
		
		$N$ will be called \emph{exchange compatible} if it is compatible with all the exchange pairs in $\mathcal C$.
	\end{defi}
		
	We denote by $\ens{e_i, i \in Q_0}$ the canonical basis of $\Z^{Q_0}$.
	For any object $M$ in $\mathcal C$, we denote by 
	$$\ddim \Hom_{\mathcal C}(T,M)=\sum_{i \in Q_0} \dim \Hom_{\mathcal C}(T_i,M)e_i$$
	
	\begin{theoremmain}
		Let $Q$ be an acyclic quiver, $\mathcal C$ its cluster category and $T$ a cluster-tilting object in $\mathcal C$. Assume that for any $i \in Q_0$, $T_i$ is exchange compatible, then for any indecomposable exceptional object $M$ of $\mathcal C$
		$$\delta_T(X^T_M)=\left\{\begin{array}{rl}
			-e_i & \textrm{ if } M\simeq ST_i\\
			\ddim \Hom_{\mathcal C}(T,M)& \textrm{ otherwise}
		\end{array}\right.$$
	\end{theoremmain}
	
	\begin{proof}
		We fix $i \in Q_0$. Proving the theorem is equivalent to proving that for any indecomposable exceptional object $M$, we have 
		$$\delta_T(X^T_M)_i=\dim \Hom_{\mathcal C}(T_i,M)-m_i(M)=:h_i(M)$$
		where $m_i$ is the multiplicity of $ST_i$ as a direct summand of $M$ and $\delta_T(X^T_M)_i$ denotes the $i$-th component of $\delta_T(X^T_M)$.
	
		Fix $M$ an indecomposable exceptional object, it can be completed into a cluster-tilting object in $\mathcal C$. The tilting graph being connected, it suffices to prove by induction that if the property holds for all the direct summands of a cluster-tilting object $R$, then it holds for all the direct summands of any cluster-tilting object $R'$ joined to $R$ by an edge in the tilting graph. We start the induction with the cluster-tilting object $\bigoplus_{i=1}^q ST_i$ for which the result clearly holds. 
	 	
	 	Fix now $R$ and $R'$ two cluster-tilting objects joined by an edge in the tilting graph. We can write $R=U \oplus \overline U$ and $R'=U^* \oplus \overline U$ with $(U,{U^*})$ an exchange pair. We write $B=E_{U^*,U}$ and $B'=E_{U,U^*}$. Corollary \ref{corol:exchangemult} implies that
	 	$$X^T_UX^T_{{U^*}}=X^T_B + X^T_{B'}$$
	 	and by lemmas \ref{lem:XTMweak} and \ref{lem:faiblementpositif}, we thus have 
	 	$$\delta_T(X^T_U)+\delta_T(X^T_{U^*})=\max(\delta_T(X^T_B), \delta_T(X^T_{B'}))$$
	 	
	 	If $U \not \simeq ST_i \not \simeq U^*$, as $T_i$ is compatible with the exchange pair $(U,U^*)$, we have
	 	$$
	 	\dim \Hom_{\mathcal C}(T_i,U)+\dim \Hom_{\mathcal C}(T_i,U^*)=\max(\dim \Hom_{\mathcal C}(T_i,B),\dim \Hom_{\mathcal C}(T_i,B')) 
	 	$$
	 	As $U$ and $U^*$ are indecomposable, we have $m_i(U)=m_i(U^*)=0$ and thus 
	 	$$h_i(M)=\dim \Hom_{\mathcal C}(T_i,M)$$
	 	Then, the above equality becomes
	 	$$h_i(U)+h_i(U^*)=\max(\dim \Hom_{\mathcal C}(T_i,B),\dim \Hom_{\mathcal C}(T_i,B'))$$
	 	By induction, $\delta_T(X^T_U)=h_i(U)$, $\delta_T(X^T_B)=h_i(B)$ and $\delta_T(X^T_{B'})=h_i(B')$, thus 
	 	$$\delta_T(X^T_{U^*})+h_i(U)=\max(h_i(B), h_i(B'))$$
	 	so we only have to prove that 
	 	\begin{equation}\label{equation:exchange2}
	 		\max(\dim \Hom_{\mathcal C}(T_i,B),\dim \Hom_{\mathcal C}(T_i,B'))=\max(h_i(B),h_i(B')) 
	 	\end{equation}
	 	Assume that $ST_i$ is not a direct summand of $\overline U$, it follows that $m_i(B)=m_i(B')=0$ and the result holds. Assume now that $ST_i$ is a direct summand of $\overline U$, it follows that $\Ext^1_{\mathcal C}(ST_i,B)=0$ and thus $\Hom_{\mathcal C}(T_i,B)=0$ and similarly $\Hom_{\mathcal C}(T_i,B')=0$. Now, equation (\ref{equation:exchange2}) holds if and only if $m_i(B)=0$ or $m_i(B')=0$. But it is known (see \cite{BMR2}) that $B \oplus B'$ is a basic object, thus $B$ and $B'$ have no common direct summands and thus either $m_i(B)=0$ or $m_i(B')=0$, which gives the induction step.

		Now, if $U \simeq ST_i$, then $\delta_T(X^T_{U})_i=-1$. As $U \simeq ST_i$ and $B,B' \in \add \overline U$, we have $\Ext^1_{\mathcal C}(B,ST_i)=0=\Ext^1_{\mathcal C}(B',ST_i)$ and $\Ext^1_{\mathcal C}(ST_i, U^*)\simeq \Hom_{\mathcal C}(T_i,U^*) \simeq k$. Also, $m_i(B)=0=m_i(B')$ and thus $h_i(B)=h_i(B')=0$. 
		Now 
		\begin{align*}
		 	\delta_T(X^T_{U^*})_i 
		 		&=\max(\delta_T(X^T_B)_i,\delta_T(X^T_{B'})_i)-\delta_T(X^T_{U})_i\\
		 		&=1\\
		 		&=\dim \Hom_{\mathcal C}(T_i,U^*)
		\end{align*}
		As $U^*$ is indecomposable and non-isomorphic to $U$, it follows that $m_i(U^*)=0$ and thus 
		$$\delta_T(X^T_{U^*})_i =h_i(U^*).$$
		
		Now if $U^* \simeq ST_i$, then clearly $\delta_T(X^T_{U^*})_i=\delta_T(x_i)=-1=h_i(U^*)$ and the theorem is proved.
	\end{proof}

		We now claim that theorems \ref{theorem:main} and \ref{theorem:mainBMR3} are equivalent. Indeed, fix $x$ a cluster variable in $\mathcal A(Q_T)$, then as the following diagram commutes
		$$\xymatrix{ 
			\mathcal T \ar[rd]_{X^T_?} & \ar[l]^{\alpha} \Cl(Q)\\
				& \Cl(Q_T)\ar[u]^{\Phi_T}\ar[r]^{\delta_T} & \Z^{Q_0}
		}$$
		we have $\alpha(\Phi_T(x))=M$ for some indecomposable exceptional object $M \in \mathcal T$ such that $X^T_M=x$. Now if $x \neq x_i$ for all $i \in Q_0$, it follows that $x$ has a $T$-denominator if and only if
		\begin{align*}
			\ddim \Hom_{\mathcal C}(T,M) 
				&= \delta_T(x)\\
				&= \delta_T(X^T_M)
		\end{align*}
		and thus the theorems are equivalent. And if $x=x_i$, then $M=ST_i$ and $x$ has a $T$-denominator.
	
\end{section}

\begin{section}{Proof of theorem \ref{theorem:main2}}\label{section:main2}
	We now prove theorem \ref{theorem:main2}. In order to give an alternative proof of this theorem, we will nevertheless have to rely on some part of the work presented in \cite{BMR3}. More precisely, we will use the proposition 2.4 of \cite{BMR3}:
	
	\begin{prop}[\cite{BMR3}]\label{propBMR3}
		If $N$ is an indecomposable exceptional object such that $\End_{\mathcal C}(N) \not \simeq k$, then there is an indecomposable exceptional object $N^*$ such that $(N,N^*)$ is an exchange pair and $N$ is not compatible with this exchange pair.
	\end{prop}
	
	We now give a proof for theorem \ref{theorem:main2} which is an alternative proof of point (b) of theorem 1.5 in \cite{BMR3}:
		
	\begin{theoremmaintwo}
		Let $Q$ be a finite quiver with no oriented cycles and $\mathcal C$ be the cluster category associated to $kQ$. Let $T=\bigoplus_{i=1}^q T_i$ be a cluster-tilting object in $\mathcal C$. Let
		$\mathcal A(Q)$ be the cluster algebra associated to $Q$. If every cluster variable of $\mathcal{A}$ has a $T$-denominator, then $\End_{\mathcal C}(T_i)\simeq k$ for all $i$.
	\end{theoremmaintwo}

  	\begin{proof}
		Assume that $\End_{\mathcal C}(T_i)\not \simeq k$ for some $T_i$. It follows from proposition \ref{propBMR3} that there exists an indecomposable exceptional object $T_i^*$ such that $(T_i,T_i^*)$ is an exchange pair and $T_i$ is not compatible with respect to this exchange pair. 
		We write $\overline T$ the complement such that $T_i \oplus \overline T$ and $T_i^* \oplus \overline T$ are cluster-tilting objects, then $(T_i,T_i^*)$ is an exchange pair and we write $B=E_{T_i^*,T_i}$ and $B'=E_{T_i,T_i^*}$.
  		
  		Corollary \ref{corol:exchangemult} implies that 
  		$$X^T_{T_i}X^T_{T_i^*}=X^T_B+X^T_{B'},$$
  		and lemmas \ref{lem:XTMweak} and \ref{lem:faiblementpositif} give
  		$$\delta_T(X^T_{T_i})+\delta_T(X^T_{{T_i^*}})=\max(\delta_T(X^T_B),\delta_T(X^T_{B'}))$$
  		
  		As $\Ext^1_{\mathcal C}(T_i, ST_i) \neq 0$, $ST_i$ is not a direct summand of $T_i \oplus \overline T$. We can thus assume that $\delta_T(X^T_{M})_i=\ddim \Hom_{\mathcal C}(T_i,M)$ for any direct summand $M$ of $\overline T$ and where $\delta_T(X^T_{M})_i$ denotes the $i$-th component of the dimension vector $\delta_T(X_M^T)$. As $B$ and $B'$ are in $\add \overline T$, we have in particular that $\delta_T(X^T_B)_i=\ddim \Hom_{\mathcal C}(T_i,B)$ and $\delta_T(X^T_{B'})_i=\ddim \Hom_{\mathcal C}(T_i,B')$.
  		
  		Now, $T_i$ is not compatible with the exchange pair $(T_i,T_i^*)$. Moreover, $\tau T_i \not \simeq T_i$ (otherwise $T_i$ has a self-extension) and $\tau T_i \not \simeq T_i^*$ (because $\Ext^1_{\mathcal C}(T_i,T_i^*)\simeq k$ and $\Ext^1_{\mathcal C}(T_i, \tau T_i)\simeq \Hom_{\mathcal C}(ST_i,ST_i) \simeq \End_{\mathcal C}(T_i) \neq k$). It thus follows from the definition that 
  		$$\dim \Hom_{\mathcal C}(T_i,T_i)+\dim \Hom_{\mathcal C}(T_i,{T_i^*}) \neq \max(\dim \Hom_{\mathcal C}(T_i,B),\dim \Hom_{\mathcal C}(T_i,B'))$$
  		
  		Then, 
  		\begin{align*}
  			\delta_T(X^T_{T_i^*}) 
  				&= \max(\delta_T(X^T_B),\delta_T(X^T_{B'})) - \delta_T(X^T_{{T_i}})\\
  				&= \max(\ddim \Hom_{\mathcal C}(T_i,B),\ddim \Hom_{\mathcal C}(T_i,B')) - \ddim \Hom_{\mathcal C}(T_i,T_i)\\
  				&\neq \ddim \Hom_{\mathcal C}(T_i,T_i^*)
  		\end{align*}
  		But $T_i^*$ is indecomposable exceptional, so $X^T_{T_i^*}$ is a cluster variable in $\mathcal C$ which has no $T$-denominator.
  	\end{proof}
  	
	Note that in \cite{BMR3}, the authors proved theorems \ref{theorem:mainBMR3} and \ref{theorem:main2} for $k$ an arbitrary algebraically closed field whereas our proofs restrict to the case where $k=\C$ is the field of complex numbers.
\end{section}

\section*{Acknowledgements}
	The author would like to thank Idun Reiten for corrections and interesting remarks concerning the subject. He would also like to thank the coordinators of the Liegrits network for organizing his stay at the NTNU in Trondheim where this paper was written.


\end{document}